\documentclass[letterpaper, 10 pt, conference]{ieeeconf}  

\IEEEoverridecommandlockouts                              
\usepackage{amssymb}  

\newtheorem{thm}{Theorem}
\newtheorem{proposition}[thm]{Proposition}

\newtheorem{remark}{Remark}
\newtheorem{definition}{Definition}
\newtheorem{assumption}{Assumption}

\title{
Small time reachable set of bilinear quantum systems
}

\author{\authorblockN{Nabile Boussa\"{i}d}
\authorblockA{Laboratoire de math\'ematiques\\
Universit\'e de Franche--Comt\'e\\
25030 Besan\c{c}on, France\\
{\tt\small Nabile.Boussaid@univ-fcomte.fr}}
\and
\authorblockN{Marco Caponigro}
\authorblockA{Dept of Mathematical Sciences and CCIB\\
Rutgers University\\
08102 Camden, NJ, USA\\
{\tt\small marco.caponigro@rutgers.edu}}
\and
\authorblockN{Thomas Chambrion}
\authorblockA{Institut \'Elie Cartan de Nancy and\\
INRIA Nancy Grand Est\\
54506 Vand{\oe}uvre, France\\
{\tt\small Thomas.Chambrion@inria.fr}}
}

\begin{document}

\maketitle
\thispagestyle{empty}
\pagestyle{empty}

\begin{abstract}
This note presents an example of bilinear conservative system in an infinite dimensional Hilbert space for which approximate controllability in the Hilbert unit sphere holds for arbitrary small times. This situation is in contrast with the finite dimensional case and is due to the unboundedness of the drift operator. 
\end{abstract}

\section{INTRODUCTION}

\subsection{Control of quantum systems}
The state of a quantum system evolving in a Riemannian manifold $\Omega$ is described by its \emph{wave function}, a point $\psi$ in $L^2(\Omega, \mathbf{C})$. When the system is 
submitted to an electric field (e.g., a laser), the time evolution of the wave is given, under the dipolar approximation and neglecting decoherence,  by the Schr\"{o}dinger equation:
\begin{equation}
\mathrm{i} \frac{\partial \psi}{\partial t}=(-\Delta  +V(x)) \psi(x,t) +u(t)  W(x) \psi(x,t)
\end{equation}
where $\Delta$ is the Laplace-Beltrami operator on $\Omega$,  $V$ and $W$ are real potential accounting for the properties of the free system and the control field respectively, while 
the real function of the time $u$ accounts for the intensity of the laser. 

It is standard to rewrite the dynamics as 
\begin{equation}\label{EQ_main}
\frac{d}{dt} \psi=(A+u(t) B) \psi
\end{equation}
where $\psi$ belongs to a separable Hilbert space $H$ and $(A,B)$ satisfies Assumption \ref{ASS_1}.
\begin{assumption}\label{ASS_1}
$A$ and $B$ are linear operators with domain $D(A)$ and $D(B)$ such that
\begin{enumerate}
\item $A$ is essentially skew-adjoint (possibly unbounded) with domain $D(A)$;
\item there exists an Hilbert basis $(\phi_k)_{k \in \mathbf{N}}$ of $H$ made of eigenvectors of $A$. For every $k$, $A \phi_k= \mathrm{i}\lambda_k \phi_k$;
\item for every $k$ in $\mathbf{N}$, $\phi_k$ belongs to the domain $D(B)$ of $B$;
\item for every $u$ in $\mathbf{R}$, $A+uB$ is essentially skew-adjoint on $D(A)\cap D(B)$;
\item $B$ is essentially skew-adjoint.
\end{enumerate}
\end{assumption}
From Assumption \ref{ASS_1}, for every $u,t$ in $\mathbf{R}$, $e^{t(A+uB)}$ is a unitary operator. By concatenation, for every piecewise constant function $u$, one can define the solution $t\mapsto \Upsilon^u_t \psi_0$ of (\ref{EQ_main}) with initial condition $\psi(0)=\psi_0$. With extra regularity hypotheses (for instance: $B$ bounded), it is possible to define $\Upsilon^u$ for controls $u$ not necessarily piecewise constant.  

A natural question, with many practical implications, is to determine the set of wave functions 
$ \Upsilon^u_T \psi_0$ that can be reached from a given initial condition $\psi_0$ at a 
given time $T$ when the control law $u$ varies in the set $\mathcal U$ of admissible controls 
(here, $\mathcal U$ is the set of piecewise constant functions). 
This set $\{\Upsilon^u_T \psi_0, u \in \mathcal{U} \}$ is called the \emph{reachable set} from 
$\psi_0$ at time $T$ and is denoted ${\mathcal R}_T(\psi_0)$. For $T \in (0,+\infty)$,  
the reachable set from  $\psi_0$ at time smaller than $T$  is 
${\mathcal R}_{< T}(\psi_0)=\cup_{t < T} {\mathcal R}_t(\psi_0)$ and  ${\mathcal R}_{< +\infty}(\psi_0)=
\cup_{T>0}{\mathcal R}_{T}(\psi_0)$.
At this time, no definitive description of the reachable sets is known, 
with the exception of ${\mathcal R}_{< +\infty}(\psi_0)$ for a few simple examples (\cite{beauchard}, \cite{camillo}) where $\mathcal U=L^2([0,+\infty),\mathbf{R})$ and $\Omega$ is 
a bounded interval of $\mathbf{R}$. Notice that even in this 1-D framework, the results are far from obvious. 

Instead of looking at the complicated structure of  ${\mathcal R}_T(\psi_0)$, one may consider its topological closure for a given norm. Many results of \emph{approximate controllability} 
have been given in the last few years. We refer to \cite{turinici,nersesyan,
beauchard-mirrahimi, mirrahimi-continuous, camillo,Schrod2} and references
therein for a description of the known theoretical results concerning the
existence of  controls steering a given source to a given target. As proved in
\cite{nersesyan, genericity-mario-paolo, genericity-mario-privat}, approximate
controllability is a generic property for systems of the type of (\ref{EQ_main}).

\subsection{Temporal diameter}

Let a couple $(A,B)$ of linear operators  associated with an Hilbert basis $(\phi_k)_{k \in \mathbf{N}}$ of $H$ made of eigenvectors of $A$ be given. If ${\mathcal R}_{< + \infty}(\phi_1) $ is everywhere dense in the Hilbert unit sphere $\mathbf{S}_H$ of $H$,   we define the \emph{temporal diameter} of $H$ for the control system (\ref{EQ_main}) as 
$$\rho=\inf \left \{T \geq 0 \left | \overline{{\mathcal R}_{< T}(\phi_1)}=\mathbf{S}_H  \right . \right \}.$$
The aim of this note is to give a positive answer (Theorem \ref{TH_main}) to the following question:
``Does it exists a couple $(A,B)$ satisfying Assumption \ref{ASS_1} such that $\rho=0$?''

The importance of this question goes beyond its purely mathematical aspect. A positive answer is a strong justification of the fact that we can neglect decoherence in (\ref{EQ_main}), since decoherence has little effect in small times.

\subsection{Content of the paper}
The first part of this paper (Section \ref{SEC_finite_dim}) is a short summary of the finite dimensional case. It includes finite dimensional estimates (Section \ref{SEC_finite_dim_estimates}) that are instrumental in our study. The second part (Section \ref{SEC_infinite_dim}) presents basic infinite dimensional material. The third part presents an example (Section \ref{SEC_Example_model}) with zero temporal diameter. The technical computation of this diameter (Sections \ref{SEC_Example_estimates} and \ref{SEC_Example_proof}) combines the results of Sections \ref{SEC_finite_dim_estimates} and \ref{SEC_wc}.

\section{FINITE DIMENSIONAL CASE}\label{SEC_finite_dim}

In this Section, we concentrate on the case where $H$ is finite dimensional. This case has been extensively studied (\cite{Brockett}, \cite{q5}). In this case, a classical choice for the set 
of admissible controls is the set $L^1_{loc}([0,+\infty),\mathbf{R})$ of locally integrable functions in $\mathbf{R}$.Other choices could be to restrict to the set of  piecewise constant 
functions or to extend to the set of Radon measures.  By continuity of the input-output mapping, and since we are only concerned with the closure of the reachable sets, all this choice 
are equivalent. 

\subsection{Bilinear systems in semi-simple Lie Groups}
The operators $A$ and $B$ can be seen as skew-hermitian matrices. Changing $A$ into $A-\mathrm{Trace}(A)/\mathrm{dim}(H)$ and $B$ into $B-\mathrm{Trace}(B)/\mathrm{dim}(H)$ 
just induces a physically irrelevant phase shifts. 

 Considering the resolvent, the original system in $H=\mathbf{C}^N$ can be lift to $SU(N)$, the group of unitary matrices of order $N$ with determinant 1. Denoting with $x$ the matrix 
of the linear operator $\psi \mapsto \Upsilon^u_t \psi$, (\ref{EQ_main}) turns into
\begin{equation}\label{EQ_lift_SU(n)}
\frac{d}{dt}x=(A+uB)x=dR_x(A+uB)
\end{equation}
with initial condition $x(0)=I_H$ where $R$ is the right translation in $SU(N)$: $R_x:y\mapsto yx$. 

The lift  of (\ref{EQ_main})  from $\mathbf{C}^N$ to $SU(n)$ as the right invariant controlled system (\ref{EQ_lift_SU(n)}) allows us to use the rich structure of semi-simple compact Lie 
groups. In particular, the Lie algebra $\mathfrak{su}(N)=T_{I_N}SU(N)$  turns into an Euclidean space when endowed with the(opposite) of the bi-invariant negative definite Killing form 
$K(x,y)=(N\pm1)\mathrm{Trace}(xy)$.

The major drawback is that we have now to consider the whole propagator $x$ (that is, equation (\ref{EQ_main}) with all possible initial conditions).

A classical result states that the temporal diameter of $SU(N)$ for the dynamic of (\ref{EQ_lift_SU(n)}) is finite for generic pairs $(A,B)$.
\begin{proposition}
 If $\mathrm{Lie}(A,B)=\mathfrak{su}(N)$, then there exists $T>0$ such that ${\mathcal R}_{<T}(\psi)=\mathbf{S}_H$ for every $\psi$ in $\mathbf{S}_H$. In other words, $\rho<+\infty$. 
\end{proposition}
\begin{proof}
 We refer for instance to \cite{such}.
\end{proof}

For every admissible control $u$, define $v:t\mapsto \int_0^t u(\tau)\mathrm{d}\tau$ and  $y:t\mapsto \exp \left (-v (t) B \right )x$. The dynamics of $y$ is given by
\begin{equation}
\frac{d}{dt}y=e^{-v(t) B}  A e^{v(t) B}y :=dR_{\mathrm{Ad}_{e^{v(t)B} }A }y
\end{equation}
The adjoint mapping $\mathrm{Ad}$ is an isometry for the Killing norm. In other words,  the derivative of $y$ has a constant Killing norm (equal to the Killing norm of $A$).  As a consequence, the temporal diameter is positive as soon as the torus $\{e^{KB}|K\in \mathbf{R}$ does not fill $SU(n)$. 
\begin{proposition}
When $H$ is finite dimensional with dimension larger than or equal to 2, $\rho>0$. 
\end{proposition}

\subsection{Some time estimates} \label{SEC_finite_dim_estimates}
In this Section, we give an estimate of the time needed to steer an eigenstate of $A$ to another. The proof relies on averaging techniques (see \cite{periodic} for details).
\begin{definition}
Let $H=\mathbf{C}^N$ and  $(A^{(N)},B^{(N)})$ satisfy Assumption \ref{ASS_1}.
A couple $(j,k)$ is a \emph{non-degenerate} transition of $(A^{(N)},B^{(N)})$ if  i) $\langle \phi_j, B^{(N)} \phi_k \rangle\neq 0$ and ii) $\lambda_j-\lambda_k=\lambda_l-\lambda_m$ implies that
 $\{j,k\}=\{l,m\}$ or $\{l,m\}\cap \{j,k\}=\emptyset$ or $\langle \phi_l, B^{(N)} \phi_m \rangle=0$.
\end{definition}

\begin{proposition}\label{PRO_cosinus}
Let $(A^{(N)},B^{(N)})$ be a couple of skew-Hermitian matrices of order $N$. Assume that $A^{(N)}$ is diagonal and that
$(1,2)$ is a non degenerate transition of $(A^{(N)},B^{(N)})$. We denote with $X^u_{(N)}$ the propagator of $x'=(A^{(N)}+uB^{(N)})x$.
Define $T=\frac{2\pi}{|\lambda_2-\lambda_1|}$ and $u^{\ast}:t\mapsto |\lambda_2-\lambda_1|\cos((\lambda_2-\lambda_1)t + \phi)$.
Then, for every $n$ in $\mathbf{N}$, for every $t$
\begin{eqnarray*}
\frac{\|X^{u^{\ast}/n}_{(N)}(t,0)-e^{tA^{(N)}}e^{K M^{\dag}}\|}{I(C+1)  \|B^{(N)}\|} \leq
\frac{ 1 + 2K \|B^{(N)}\|}{n}
\end{eqnarray*}
with $I=\int_0^T|u^{\ast}(\tau)\mathrm{d}\tau=4$, $K=\frac{1}{n}\int_0^t|u^{\ast}(t)|\mathrm{d}t$, 
$M^{\dag}$ the skew-Hermitian matrix of order $N$ which entries are all zero but the ones of index $(1,2)$ and $(2,1)$ equal to $\pi b_{12} e^{\mathrm{i} \phi} /4$ and $\pi b_{21}e^{-\mathrm{i}\phi}/4$ respectively and
$$  C=\sup_{(j,k)\in \Lambda} \left | \frac{\int_0^T u^{\ast}(\tau) e^{\mathrm{i} (\lambda_j-\lambda_k)\tau}\mathrm{d}\tau}{\sin \left ( \pi\frac{|\lambda_j-\lambda_k|}{|\lambda_2-\lambda_1|} \right )} \right |,$$
where $\Lambda$ is the set of all pairs $(j,k)$ in  $\{1,\ldots,N\}^2$  such that $b_{jk} \neq 0$ and $\{j,k\}\cap\{1,2\} \neq \emptyset$ and $ |\lambda_j-\lambda_k|\notin \mathbf{Z}|\lambda_2-\lambda_1|$.
\end{proposition}
\begin{proof}
This is a particular case (for $u^{\ast}:t\mapsto |\lambda_2-\lambda_1|\cos((\lambda_2-\lambda_1)t + \phi)$ of the inequality (13) in \cite{periodic}.
\end{proof}

\subsection{A technical computation}\label{SEC_exponential}

In order to apply Proposition \ref{PRO_cosinus}, we will have to find $K$ and $\phi$ such that $e^{K M^{\dag}}$ sends a given vector to another one.

For a given $(\alpha,\beta)^T \in \mathbf{C}^2$, we  will need $r$ and $\phi$ in $\mathbf{R}$ such that
$e^M  (\alpha,\beta)^T$ is colinear to $(0,1)^T$ with
$$M=\left ( \begin{array}{cc}0 & r  e^{\mathrm{i}\phi} \\ r  e^{-\mathrm{i}\phi} & 0 \end{array} \right ).$$

For every $s$ in $\mathbf{R}$,
$$\exp(tM)=\left ( \begin{array}{cc} \cos(rt) & e^{\mathrm{i}\phi} \sin(rt) \\ -e^{\mathrm{i}\phi}\sin(rt) & \cos(rt) \end{array} \right ).$$
For every $(\alpha,\beta)^T \in \mathbf{C}^2$,
$$\exp(tM) \left ( \begin{array}{c} \alpha\\ \beta \end{array} \right)=
\left ( \begin{array}{c} \alpha \cos(rt) +e^{\mathrm{i}\phi}  \sin(rt) \beta \\ \beta \cos(rt) -e^{-\mathrm{i}\phi} \alpha \sin(sa)  \end{array} \right).$$

In the following, we assume without lost of generality that $|\alpha|^2+|\beta|^2=1$.
There exist $\theta$ in $[0,\pi/2]$, $\alpha_1,\beta_1$ in $(-\pi,\pi]$ such that
\begin{eqnarray}
 \alpha&=&\cos \theta e^{\mathrm{i}\alpha_1}\\
\beta&=&\sin \theta e^{\mathrm{i}\beta_1}.
\end{eqnarray}
With these notations,
$\alpha \cos(rt) + e^{\mathrm{i}\phi} \beta \sin(rt)=0$ if and only if
$$
\left \{ \begin{array}{lcl}
          \cos \left (\frac{\alpha_1-\beta_1-\phi}{2} \right ) \cos(\theta-rt)&=&0\\
          \sin \left (\frac{\alpha_1-\beta_1-\phi}{2} \right ) \cos(\theta+rt)&=&0
	  \end{array}
\right.
$$
Since we are interested in small $t$, we will chose $\phi$ and $t$ such that
$$
\left \{
    \begin{array}{lcl}
	  \cos \left (\frac{\alpha_1-\beta_1-\phi}{2} \right )\!\!\! &=&0\!\!\!\\
	  \cos(\theta+rt)\!\!\!&=&0\!\!\!
    \end{array}
\right.
\mbox{that is,}
\left \{
    \begin{array}{lcl}
	  \frac{\alpha_1-\beta_1-\phi}{2}  &=&\frac{\pi}{2}(\pi)\\
	  \theta+rt&=&\frac{\pi}{2}(\pi).
    \end{array}
\right.
$$

\section{INFINITE DIMENSIONAL TOOLS} \label{SEC_infinite_dim}

\subsection{Notations}
Let $(A,B)$ satisfy Assumption \ref{ASS_1}. For every $N$ in $\mathbf{N}$, we define $\mathcal{L}_N$ the linear space spanned by $\phi_1,\phi_2, \ldots,\phi_N$ and $\pi_N:H\rightarrow H$, the orthogonal projection onto $\mathcal{L}_N$:
$$
\pi_N(\psi)=\sum_{k=1}^N \langle \phi_k, \psi \rangle \phi_k.
$$
The compressions of order $N$ of $A$ and $B$ are the finite rank operators 
$A^{(N)}=\pi_N A_{{\upharpoonright \mathcal{L}_N}}$ and
$B^{(N)}=\pi_N B_{{\upharpoonright \mathcal{L}_N}}$. The Galerkin approximation of (\ref{EQ_main}) at order $N$ is the infinite dimensional system
\begin{equation}\label{EQ_Galerkin}
 \frac{\mathrm{d}}{\mathrm{d}t}x=A^{(N)}x +u(t) B^{(N)} x.
\end{equation}
Since $\mathcal{L}_N$ is invariant by (\ref{EQ_Galerkin}), one may also consider (\ref{EQ_Galerkin}) as a finite-dimensional system, whose propagator is denoted by $X^u_{(N)}(t,s)$.

The operator $|A|$ is a positive self-adjoint operator. In the case where $A$ is injective, we define for every $k\geq 0$ the norm $k$-norm: $\|\psi\|_k=\||A|^k \psi\|$. 

\subsection{Basic facts}\label{SEC_basic_facts}

\begin{proposition}\label{PRO_tore_pas_dense}
Assume that $H$ has infinite dimension. Then, for every $\psi_0$ in $\mathbf{S}_H$, $\overline{\{e^{KB}\psi_0,K \in \mathbf{R} \}} \neq \mathbf{S}_H$ .
\end{proposition}
\begin{proof}
%
By the spectral theorem, up to a unitary transformation, $H=L^2(\Omega,\mathbf{C})$ with $\Omega$ a set of cardinality larger than two and $B$ is the multiplication by a purely imaginary function. Then, the set $\overline{\{e^{KB}\psi_0, K \in \mathbf{R}\}}$ is included in the set $\{\psi \in H \mbox{such that } |\psi(x)|=|\psi_0(x)| \mbox{for a. e.} x\in \Omega \}$. The latter set is not equal to $\mathbf{S}_H$. 
\end{proof}

\begin{proposition}\label{PRO_min_diam_inf_A_borne}
If  $H$ has infinite dimension and $A$ is bounded, then $\rho>0$.
\end{proposition}
\begin{proof}
For every $u$ in $\mathcal U$, define $Y^u:t\mapsto e^{-\int_0^t|u(\tau)|\mathrm{d}\tau B} \Upsilon^u_t$.  For every $\psi_0$ in $\mathbf{S}_H$, for every $k$ in $\mathbf{N}$, the mapping $t \mapsto Y_t^u \psi_0$  satisfies
$$
\frac{\mathrm{d}}{\mathrm{d}t}\langle \phi_k, Y^u_t\psi_0\rangle = \langle e^{-\int_0^t|u(\tau)|\mathrm{d}\tau B} A e^{\int_0^t|u(\tau)|\mathrm{d}\tau B} \phi_k, Y_t^u \psi_0 \rangle.
$$
In particular, for every $\psi_0$, for every $t$,
$$
\left |\frac{\mathrm{d}}{\mathrm{d}t}\langle \phi_k, Y^u_t\psi_0\rangle \right |  \leq \|A \|.
$$ 
Consider now two points $\psi_0$ and $\psi_1$ in $\mathbf{S}_H$ such that the distance $\delta$ between the two sets  $\overline{\{e^{KB}\psi_0,K \in \mathbf{R}\}}$ and $\overline{\{e^{KB}\psi_1, K \in \mathbf{R}\}}$ is not zero (such a couple $(\psi_0,\psi_1)$ exists from Proposition \ref{PRO_tore_pas_dense}). Then $\rho\geq \delta/\|A\|$, since $\psi_1 \notin \overline{{\mathcal{R}}_{<t}\psi_0}$ for $t<\delta/\|A\|$.
\end{proof}

\begin{proposition}\label{PRO_min_rho_eigenvecotr}
If  $H$ has infinite dimension and  $B$ admits an eigenvector in the domain of $A$, then $\rho>0$.
\end{proposition}
\begin{proof}
For every $u$ in $\mathcal U$, define as above $Y^u:t\mapsto e^{-\int_0^t|u(\tau)|\mathrm{d}\tau B} \Upsilon^u_t$.  Let $v$ be an eigenvector of $B$ associated with eigenvalue $\lambda$. For every $\psi_0$ in $\mathbf{S}_H$,  the mapping $t \mapsto Y_t^u \psi_0$  satisfies
$$
\frac{\mathrm{d}}{\mathrm{d}t}\langle v, Y^u_t\psi_0\rangle = \langle  A v, Y_t^u \psi_0 \rangle.
$$  
Fix $\psi_0,\psi_1$ in $\mathbf{S}_H$ such that $|\langle v, \psi_0 \rangle | \neq |\langle v, \psi_1 \rangle |$. Then $\rho>\delta/\|Av\|$ since  $\psi_1 \notin \overline{{\mathcal{R}}_{<t}\psi_0}$ 
for $t<\delta/\|Av\|$.
\end{proof}

\begin{remark}
 Propositions \ref{PRO_tore_pas_dense}, \ref{PRO_min_diam_inf_A_borne} and \ref{PRO_min_rho_eigenvecotr} are true also when $H$ has finite dimension larger than or equal to two.
\end{remark}

\subsection{The RAGE theorem}\label{SEC_RAGE}
We define $\mathcal{L}_B$ the linear space spanned by the eigenvectors (if any) of $B$. 
We define $H_{B}=\mathcal{L}_B^{\perp}$
\begin{thm}\label{TH_RAGE}
Let $N$ in $\mathbf{N}$ and $\psi_0$ in $H_{B}$.  Then there exists a sequence 
$(k_n)_{n\in \mathbf{N}}$ with limit $+\infty$ such that $\| \pi_N e^{k_n B} \psi_0 \|$ tends to zero as $n$ tends to infinity.
\end{thm}
\begin{proof}
This is a weak-version of the celebrated RAGE-theorem, see \cite[Theorem XI.115]{reed-simon-3}.
\end{proof}

\subsection{Weakly-coupled quantum systems}\label{SEC_wc}
 
\begin{definition}
Let $k$ be a positive number and let  $(A,B)$ satisfy Assumption \ref{ASS_1}
 and such that the spectrum of $\mathrm{i}A$ is purely discrete $(\lambda_k)_k$  and tends to $+\infty$ 
Then $(A,B)$ is
\emph{$k$ weakly-coupled}
if for every $u_1 \in \mathbf{R}$, $D(|A+u_1B|^{k/2})=D(|A|^{k/2})$ and
 there exists
a constant $c_{(A,B)}$ such that, for every $\psi$ in $D(|A|^k)$, $ |\Re \langle |A|^k
\psi,B \psi \rangle |\leq c_{(A,B)} |\langle |A|^k \psi, \psi \rangle|$.
\end{definition}
The notion of weakly-coupled systems is closely related to the growth of the
 $|A|^{k/2}$-norm $\langle |A|^k \psi, \psi \rangle$. For $k=1$, this quantity is the expected value of the energy of the system.
\begin{proposition}\label{PRO_croissance_norme_A} Let  $(A,B)$ be $k$-weakly-coupled.  Then,
for every $\psi_{0} \in D(|A|^{k/2})$, $K>0$,
$T\geq 0$, and $u$ in $L^1([0,\infty))$  for which
$\|u\|_{L^1}< K$, one has
$\left\|\Upsilon^{u}_{T}(\psi_{0})\right\|_{k/2} \leq
e^{c(A,B) K} \| \psi_0 \|_{k/2}.$
\end{proposition}
\begin{proof}
 This is \cite[Proposition 2]{weakly-coupled}.
\end{proof}
\begin{proposition}\label{prop:gga}
Let $k$ and $s$ be non-negative  numbers with
$0\leq s <k$. Let $(A,B)$ 
 be $k$ weakly-coupled
Assume that there
exists $d>0$, $0\leq r<k$ such that $\|B\psi \|\leq d \|\psi \|_{r/2}$ 
 for
every $\psi$ in $D(|A|^{r/2})$.
Then
for every $\varepsilon > 0 $, $K\geq 0$, $n\in \mathbf{N}$, and
$(\psi_j)_{1\leq j \leq n}$ in $D(|A|^{k/2})^n$
there exists $N \in \mathbf{N}$
such that
for every piecewise constant function $u$
$$
\|u\|_{L^{1}} < K \Rightarrow\| \Upsilon^{u}_{t}(\psi_{j}) -
X^{u}_{(N)}(t,0)\pi_{N} \psi_{j}\|_{s/2} < \varepsilon,
$$
for every $t \geq 0$ and $j=1,\ldots,n$.
\end{proposition}
\begin{proof}
 This is \cite[Proposition 4]{weakly-coupled}.
\end{proof}
\begin{remark}
 An interesting feature of Propositions \ref{PRO_croissance_norme_A} and \ref{prop:gga} is the fact that the bound  of the 
$|A|^{k/2}$ norm of the solution of (\ref{EQ_main}) or the bound on the error between the infinite dimensional 
system and its finite dimensional approximation only depend on the $L^1$ norm of the control, not on the time. 
\end{remark}
\begin{proposition}
Let $(A,B)$ be $k$-weakly coupled for some $k>0$. Then, for every $\psi_0$ in $\mathbf{S}_H$, for every $T>0$, $\{e^{KB}\psi_0, K \in \mathbf{R}\} \subset \overline{{\mathcal R}_T(\psi_0)}$.
\end{proposition}
\begin{proof}
Fix $K$ in $\mathbf{R}$ and $\varepsilon>0$. For every $\eta>0$, consider the control $u_{\eta}$, constant equal to $K/\eta$ on the time interval $[0,\eta]$ and equal to zero elsewhere. By Proposition \ref{prop:gga}, there exists $N$ such that $\|\Upsilon^{u_{\eta}}_{t}(\psi_0) -
X^{u_\eta}_{(N)}(t,0)\pi_{N} \psi_{0}\| < \varepsilon$ for every $t\geq 0$. The classical theory of ODE ensures that $X^{u_\eta}_{(N)}(\eta,0)\pi_{N} \psi_{0}$ tends to
$e^{KB}\pi_N \psi_0$ as $\eta$ goes to zero.
\end{proof}

\section{AN EXAMPLE OF APPROXIMATE CONTROLLABILITY IN TIME ZERO}\label{SEC_Example}

\subsection{A toy model}\label{SEC_Example_model}

We consider the following bilinear control system
\begin{equation}\label{EQ_Toy_model}
\mathrm{i} \frac{\partial \psi}{\partial t} = -|\Delta|^{\alpha}\psi + u(t) \cos\theta \psi \quad \theta \in \Omega
\end{equation}
where $\alpha$ is a real constant, $\Omega=\mathbf{R}/{2\pi}$ is the one dimensional torus, $H=L^2(\Omega,\mathbf{C})$ and $\Delta$ is the Laplace-Beltrami operator on $\Omega$.
\begin{remark}
  A realistic (and widely used) model for a rotating molecule is (\ref{EQ_Toy_model}) with $\alpha=1$. For $\alpha\neq 1$, the presented example is purely academic.
\end{remark}

The  self-adjoint operator $-\Delta$ has purely discrete spectrum $\{k^2,k \in \mathbf{N}\}$. All its eigenvalues are double but zero which is simple. The eigenvalue zero is associated 
with the constant
functions. The eigenvalue $k^2$ for $k>0$ is associated with the two eigenfunctions $\theta \mapsto\frac{1}{\sqrt{\pi}} \cos(k \theta)$ and $\theta \mapsto \frac{1}{\sqrt{\pi}} \sin(k 
\theta)$. The Hilbert space $H=L^2(\Omega,\mathbf{C})$ splits in two subspaces $H_e$ and $H_o$, the spaces of even and odd functions of $H$ respectively. The spaces $H_e$ and 
$H_o$ are stable under the dynamics of (\ref{EQ_Toy_model}), hence no global controllability is to be expected in $H$.

We consider the restriction of (\ref{EQ_Toy_model}) to the space $H_o$. The function $\phi_k:\theta \mapsto \sin(k \theta)/\sqrt{\pi}$ is an eigenvector of the skew-adjoint operator 
$A=\mathrm{i}|\Delta_{|H_o}|^{\frac{5}{2}}$ associated with eigenvalue $\mathrm{i}k^{2\alpha}$. The familly $(\phi_k)_{k\in \mathbf{N}}$ is an Hilbert basis of $H_o$. Here, $B$ is the restriction to $H_o$ of the multiplication by 
$-\mathrm{i}\cos(\theta)$. The skew-adjoint operators $B$ is bounded and has no eigenvalue: $\mathcal{L}_B=H_o$.
For every $j,k$, $\langle \phi_j, B \phi_k\rangle=0$ if $j=k$ or $|j-k|\geq 2$, and $\langle \phi_j, B \phi_{j+1}\rangle= -\mathrm{i}/2$.

\begin{thm}\label{TH_main}
If $\alpha>5/2$, then for every $\psi_0,\psi_1$ in the  Hilbert unit sphere of $H_o$, for every $\varepsilon>0$, for every $T>0$, there exists a piecewise constant function $u:[0,T]\rightarrow \mathbf{R}$ such that $\|\Upsilon^u_{T} \psi_0 -\psi_1 \|<\varepsilon$. In other words, for (\ref{EQ_Toy_model}), if $\alpha>5/2$ then $\rho=0$.
\end{thm}

\subsection{Some time estimates}\label{SEC_Example_estimates}

\begin{proposition}\label{PRO_majoration_temps}
Assume $\alpha>5/2$. Let $N_0$ large enough, $P\geq N_0$ and $\psi_0$, $\psi_1$ be in ${\mathcal L}_P \cap {\mathcal L}_{N_0-1}^{\perp}$ such that $\|\psi_0\|=\|\psi_1\|=1$.
 Then, for every $\varepsilon>0$, there exists $u$ such that (\ref{EQ_Toy_model}) steers $\psi_0$ to an
 $\varepsilon$-neighbourhood of $\psi_1$ in time less than 
$$ \frac{604}{\alpha^2  \varepsilon (2\alpha-5)} \frac{1}{(N_0-1)^{2\alpha-4}} + \frac{2\pi}{N_0^{2\alpha}}.$$
\end{proposition}
\begin{proof}
From Proposition \ref{prop:gga}, there exists $N_1$ in $\mathbf{N}$ such that, for every $u$, $\|u\|_{L^1}\leq 8 P$ implies that $\|X^u_{(N_1)}(t,0)\psi_j -\Upsilon^u_t \psi_j \| \leq \varepsilon/2$ for $j=0,1$. The problem is now to give an upper bound on the time needed to steer $\psi_0$ to $\psi_1$ with system 
$$
\frac{\mathrm{d}\psi}{\mathrm{d}t}=A^{(N_1)} \psi + u(t) B^{(N_1)} \psi.
$$
The idea is to find a control $u_1$ that steers $\psi_0$ to a neighborhood of $e^{\mathrm{i}\vartheta_0} \phi_{N_0}$ for some $\vartheta_0$ in $\mathbf{R}$ in time $T_0$ and, similarly, a control $u_2$ that steers $\psi_1$ to a neighborhood of $e^{\mathrm{i}\vartheta_1} \phi_{N_0}$ for some $\vartheta_1$ in $[\vartheta_0,\vartheta_0+2\pi)$ in time $T_1$. The final control is the concatenation of control $u_1$, of control $0$ during time $|\vartheta_1-\vartheta_0|/\lambda_{N_0}$ and finally of the time-reverse of $u_2$ (which steers $e^{\mathrm{i}\theta_1} \phi_{N_0}$ to a neighborhood of  $\psi_1$). The total duration is $T_0+ |\vartheta_1-\vartheta_0|/\lambda_{N_0} +T_1$.
What remains to do is to give an estimate for $T_0$ (the same computation gives a similar result for $T_1$). 

Let us describe intuitively our method. To induce the transition from $\psi_0$ to $\phi_{N_0}$, we will first put the $L^2$ mass of the $P^{th}$ level of $\psi_0$ to level $(P-1)$, without changing the modulus of the other coordinates of $\psi_0$. Then, we put all the $L^2$ mass of the  $(P-1)^{th}$ level to level $P-2$, and so on until all the mass is concentrated, up to a small error, on level $N_0$.

We give now a formal description of the above method.
Write $\psi_0=\sum_{N=N_0}^{P} x_N \phi_N$. For $N=N_0,\ldots,P$, we define
$$
u^{\ast}:t\mapsto (\lambda_{N+1}-\lambda_N) \cos(|\lambda_{N+1}-\lambda_N|t +\phi_N)
$$
$$
\theta_N=\arctan \left ( \frac{|x_N|}{\sqrt{\sum_{j=N+1}^P |x_j|^2}} \right ), \quad K_N=4 \frac{\pi-2\theta_N}{\pi},
$$
$$
C_N=\frac{\left |\int_0^{\frac{2\pi}{(N+1)^{2\alpha}-N^{2\alpha}}} u^{\ast}(t) e^{\mathrm{i}(((N+2)^{2\alpha}-(N+1)^{2\alpha})t}\mathrm{d}t \right |}{\left |\sin \left (\pi  \frac{(N+2)^{2\alpha}-(N+1)^{2\alpha}}{(N+1)^{2\alpha}-N^{2\alpha}}\right ) \right |},
$$
$$
\eta_N=\frac{N_0 \varepsilon}{4N^2},
n_N=\frac{4(1+2 K_N \|B^{(N_1)}\|) (C_{N}+1)\|B^{(N_1)}\|}{\eta_N},
$$
$$
\tau_N=\frac{\pi K_N n_N}{2|\lambda_{N+1}-\lambda_N| }\leq \frac{8\pi (1+2 K_N \|B\|)(C_N +1)\|B\| }{(\lambda_{N+1}-\lambda_N)\eta_N}
$$
We proceed by induction on $N$ from $P$ to $N_0$ to steer $\psi_0=\psi_0^P$ to $\psi_0^{P-1}$ such that $\|(1-\pi_{P-1})\psi_0^{P-1}\| + \|\pi_{N_0-1}\psi_0^{P-1}\| \leq \eta_P$. If $\psi_0^{N}$ is constructed, we
 steer  $\psi_0^N$ to $\psi_0^{N-1}$ 
such that $\|(1-\pi_N)\psi_0^N\|^2 + \| \pi_{N_0-1} \psi_0^N\|^2 \leq   1-\left (1-\sum_{l\geq N} \eta_l \right )^2$. At step $P-N_0$, $\psi_{N_0}^{N_0+1}$ is 
in an $\sum_{N_0}^P {2\eta_N} \leq {\varepsilon}/2$ neighborhood of
 $e^{\mathrm{i}\vartheta_0}\phi_{N_0}$ for some $\vartheta_0$ in $\mathbf{R}$. 

This construction is done, following Proposition \ref{PRO_cosinus}, by using the
 control $t\mapsto u^{\ast}(t)/{n_N} $ during time $\tau_N$.
The $L^1$ norm of the control (equal to $K_N$) is chosen (using the computation of Section \ref{SEC_exponential}) in such a way that $\exp \left (K_N M^{\dag} \right )$ sends the mass of the $(N+1)^{th}$ eigenstate to the $N^{th}$ one, that is a rotation of angle $\theta_N$. Since, for every $N$, $b_{N,N+1}=-\mathrm{i}/2$, the matrix $M^{\dag}$ appearing in Proposition \ref{PRO_cosinus} is a block diagonal matrix with every block equal to zero, but the one in position $N,N+1$ equal to 
$$
\left ( \begin{array}{cc} 0 & -\mathrm{i}e^{\mathrm{i}\varphi_N} \pi/8\\
                                                           \mathrm{i}e^{-\mathrm{i}\varphi_N} \pi/8         & 0
                                                                  \end{array}
 \right  ).$$
 The phase
 $\phi_N$ is equal to $a_N-b_N+\pi$ where $a_N$ and $b_N$ are the phase of coordinates $N+1$ and $N$ of $\psi_0^{N+1}$. 

Notice that, while Proposition \ref{PRO_cosinus} is stated for $n$ an integer, $n_N$ is not an integer in general. Nevertheless we can apply  Proposition \ref{PRO_cosinus} with the integer part $\lfloor n_N \rfloor$  in order to  ensure that the transformation is done
  with an error less than $\eta_N$. Changing the integer $\lfloor n_N \rfloor$ to the real number $n_N$  does not change the bound on the error for more than an factor two for large $n_N$. 

With our definition of $A$ and $B$, $\lambda_N=N^{2\alpha}$ and $\|B\|=\sqrt{2}/2$. Straightforward computations yield, for $N_0$ large enough,
$$
C_N +1\leq \frac{2 N}{\alpha \pi}, \quad K_N \leq 4 \quad \tau_N \leq \frac{151}{\alpha^2 N_0 \varepsilon  N^{2\alpha-4} },
$$
and, for $N_0$ large enough, 
\begin{eqnarray*}
T_0& \leq& \sum_{N=N_0}^P \tau_N\\
& \leq& \frac{151}{\alpha^2 N_0 \varepsilon } \sum_{N=N_0}^P \frac{1}{N^{2\alpha-4}}\\
& \leq &  \frac{151}{\alpha^2 N_0 \varepsilon (2\alpha-5)} \frac{2}{(N_0-1)^{2\alpha-5}}\\
&\leq & \frac{302}{\alpha^2  \varepsilon (2\alpha-5)} \frac{1}{(N_0-1)^{2\alpha-4}}.
\end{eqnarray*}
Finally, for $N_0$ large enough, the total time needed to steer $\psi_0$ to an $\varepsilon$-neighborhood of $\psi_1$ is less than
$$
\frac{604}{\alpha^2  \varepsilon (2\alpha-5)} \frac{1}{(N_0-1)^{2\alpha-4}} + \frac{2\pi}{N_0^{2\alpha}},
$$
which concludes the proof of Proposition \ref{PRO_majoration_temps}.
\end{proof}

\subsection{Proof of Theorem \ref{TH_main}}\label{SEC_Example_proof}

\begin{proof}[of Theorem \ref{TH_main}]
Let $\psi_0$, $\psi_1$ in $\mathbf{S}_H$, $\varepsilon>0$.
To prove Theorem \ref{TH_main}, it enough to prove that the system (\ref{EQ_Toy_model}) can approximately steer one point of $\{e^{K B}\psi_0, K \in \mathbf{R}\}$ to an $\varepsilon$-neighborhood of one point of $ \{e^{K B}\psi_1, K \in \mathbf{R}\}$ in arbitrary small time.

Since $\alpha>5/2$, $N_0^{4-2\alpha}$ and $N_0^{-2 \alpha}$ tend to zero as $N_0$ tends to infinity. 
Define $N_0$   in $\mathbf{N}$ such that
$$
\frac{604}{\alpha^2  \varepsilon (2\alpha-5)} \frac{1}{(N_0-1)^{2\alpha-4}} + \frac{2\pi}{N_0^{2\alpha}}< \eta.
$$
 By the RAGE-Theorem (Theorem  \ref{TH_RAGE}), there exist $K_0$ and $K_1$ in $\mathbf{R}$ such that $\|\pi_{N_0} e^{K_0 B}\psi_0\|<\varepsilon$ and  $\|\pi_{N_0} e^{K_1 B}\psi_1\|<\varepsilon$.
There exists $P$ in $\mathbf{N}$ such that $\|(1-\pi_P)  e^{K_0 B}\psi_0\|<\varepsilon$ and  $\|(1-\pi_P)  e^{K_1 B}\psi_1\|<\varepsilon$. 

By Proposition \ref{PRO_majoration_temps}, there exists a control that steers $\pi_P(1-\pi_{N_0-1}) \psi_0$ to an $\varepsilon$-neighborhood of $\pi_P(1-\pi_{N_0-1}) \psi_1$ in time less than $\eta$. This concludes the proof of  Theorem  \ref{TH_main}.
\end{proof}

\section{CONCLUSIONS AND FUTURE WORKS}

\subsection{Conclusions}
The question of the minimal time needed to steer a quantum system from an arbitrary source to (a neighborhood of) an arbitrary target is of great importance in practice. 
This note presents a simple example of bilinear conservative control system in an infinite dimensional Hilbert space for which approximate controllability in the Hilbert unit sphere holds for arbitrary small time. 

\subsection{Future Works}
At this time, we have no simple criterion to decide in the general case (unbounded drift operator $A$ and control operator $B$ without eigenvalue) whether a system of type (\ref{EQ_main}) has zero temporal diameter. New methods will likely be needed for further investigations.

\section{ACKNOWLEDGMENTS}

This work has been partially supported by INRIA Nancy-Grand Est.

Second and third authors were partially supported by
French Agence National de la Recherche ANR ``GCM'' program
``BLANC-CSD'', contract number NT09-504590. The
third author was partially supported by European Research
Council ERC StG 2009 ``GeCoMethods'', contract number
239748.

\bibliographystyle{IEEEtran}
\bibliography{biblio}

\end{document}